\documentclass[12pt]{article}
\usepackage{lineno}
\usepackage{orcidlink}
\usepackage{amsmath}     
\usepackage{amssymb}     
\usepackage{amsthm}
\usepackage{amsfonts}
\usepackage{hyperref}
\usepackage{comment}
\usepackage[font=small]{caption}
\usepackage{tikz}
\usepackage[percent]{overpic}
\usepackage[showonlyrefs]{mathtools}
\mathtoolsset{showonlyrefs}
\numberwithin{equation}{section}
\def\re{\operatorname{Re}}
\def\im{\operatorname{Im}}

\def\meas{\operatorname{meas}}
\def\sing{\operatorname*{sing}}

\def\length{\operatorname*{length}}

\def\Z{\mathbb{Z}}
\def\R{\mathbb{R}}
\def\C{\mathbb{C}}
\def\H{\mathbb{H}}

\def\B{\mathcal{B}}
\newcommand*{\defeq}{\mathrel{\vcenter{\baselineskip0.5ex \lineskiplimit0pt
                        \hbox{\scriptsize.}\hbox{\scriptsize.}}}%
        =}

\newtheorem{theorem}{Theorem}[section]
\newtheorem{lemma}[theorem]{Lemma}
\newtheorem{proposition}[theorem]{Proposition}
\newtheorem{corollary}[theorem]{Corollary}
\theoremstyle{definition}

\theoremstyle{remark}
\newtheorem{remark}[theorem]{Remark}

\begin{document}  
\title{On the Hausdorff dimension of the set of non-escaping points in the Julia set}
\author{Walter Bergweiler\orcidlink{0000-0001-5345-1831}, 
Weiwei Cui\orcidlink{0000-0003-2505-4710}
and Lingrui Wang\orcidlink{0009-0007-4360-7338}\thanks{The first two authors acknowledge partial support by a grant from Shandong Province.
Cui was partially supported by NSFC (No.\ 12401105), Qingdao NSF (No.\ 24-4-4-zrjj-8-jch)
and Shandong Provincial Natural Science Fund for Excellent Young Scientists Program 
(Overseas) (No.\ 2025HWYQ-021). 
Wang was supported by Shandong Postdoctoral Science Foundation (Grant No.\ SDZZ-ZR-202501290).}}
\date{}
\maketitle
\begin{center}
\emph{Dedicated to the memory of Robert L.\ Devaney}
\end{center}
\begin{abstract}
For an entire function $f$ and a non-zero complex number $\lambda$, let $f_\lambda(z)=f(\lambda z)$.
We give conditions on $f$ which imply that the Hausdorff dimension of the set of
non-escaping points in the Julia set of $f_\lambda$ tends to $1$ as $|\lambda|\to 0$.
In fact, we give an upper bound for this dimension in terms of $\lambda$.
This generalizes earlier results concerned with the case that $f(z)=\exp z$.

\smallskip
Keywords: Julia set, Fatou set, escaping set, Hausdorff dimension, disjoint type, Eremenko--Lyubich class
\end{abstract}
\section{Introduction and results}\label{sec:intro}
The main objects studied in transcendental dynamics are,
for a transcendental entire function~$f$,
the \emph{Fatou set} $F(f)$, which is the set of
all points in $\C$ where the iterates $f^n$ form a normal family,
and the \emph{Julia set} $J(f)\defeq \C\setminus F(f)$.
A major role in the theory is also played by the set $I(f)$ consisting of those points $z\in\C$ 
for which $|f^n(z)|\to\infty$ as $n\to\infty$.
This set was introduced by Eremenko~\cite{Eremenko1989} and named
\emph{escaping set} by Devaney~\cite{Devaney1989}.
We refer to~\cite{Bergweiler1993,Schleicher2010} for 
an introduction to transcendental dynamics, 
and to~\cite{Bergweiler2025a} for a
thorough discussion of the escaping set. 


The study of the dynamics of the exponential family $E_\lambda(z)\defeq \lambda e^z$ 
was pioneered by Devaney in the 1980s; see
\cite{Devaney1984a,Devaney1984,Devaney1985,Devaney1987,Devaney1984b}.
From the large further literature on exponential dynamics
we only mention~\cite{Devaney2010,Rempe2006,Schleicher2003} here.

McMullen \cite[Theorems 1.2 and~1.3]{McMullen1987} proved that $\dim J(E_\lambda)=2$ for all 
$\lambda\in\C\setminus\{0\}$, while
$\meas J(E_\lambda)=0$ if $E_\lambda$ has an attracting periodic point.
Here $\dim(\cdot)$ and $\meas(\cdot)$ denote the Hausdorff dimension and
the Lebesgue measure. McMullen actually showed that 
$\dim I(E_\lambda)=2$ and that $I(E_\lambda)\subset J(E_\lambda)$ for all 
$\lambda\in\C\setminus\{0\}$.

Urba\'nski and Zdunik~\cite[Theorem 6.1]{Urbanski2003}
proved that if $E_\lambda$ has an attracting periodic point, then
$\dim (J(E_\lambda)\setminus I(E_\lambda))<2$.
Their proof uses the thermodynamical formalism.
They also gave a direct proof~\cite[Theorem 7.2]{Urbanski2003} that
\begin{equation} \label{a3}
\lim_{\lambda\to 0}\dim\!\left( J(E_\lambda) \setminus I(E_\lambda) \right) =1
\end{equation}
and deduced~\cite[Corollary 7.3]{Urbanski2003}
from this that $\dim (J(E_\lambda)\setminus I(E_\lambda))<2$ for $|\lambda|<1/e$.

A quantitative version of~\eqref{a3} was obtained in~\cite{Bergweiler2021} where,
sharpening earlier estimates in~\cite{Karpinska1999a}, it was shown that
\begin{equation} \label{y5a}
\dim\!\left( J(E_\lambda) \setminus I(E_\lambda) \right) 
\leq 1+\frac{\log\log\log \frac{1}{|\lambda|}}{\log\log \frac{1}{|\lambda|}}
\end{equation}
if $|\lambda|$ is sufficiently small.
As shown in~\cite[Theorem~1]{Wang2004} we actually have
\begin{equation} \label{y6}
\dim\!\left( J(E_\lambda) \setminus I(E_\lambda) \right) - 1
\sim \frac{\log\log\log \frac{1}{|\lambda|}}{\log\log \frac{1}{|\lambda|}}
\end{equation}
as $|\lambda|\to 0$.
In fact, the result in~\cite{Wang2004} is more precise than~\eqref{y6} by giving a
four terms asymptotic formula for $\dim\!\left( J(E_\lambda) \setminus I(E_\lambda) \right)$
as $|\lambda|\to 0$.
Moreover, it was shown in~\cite[Theorem~2]{Wang2004} that~\eqref{y6} also holds
if the exponential function is replaced by $f(z)=a e^z+be^{-z}$, with $a,b\in\C\setminus\{0\}$.

The main purpose of this paper is to generalize~\eqref{a3} and in fact~\eqref{y5a}
to a larger class of functions that was introduced in~\cite{Bergweiler2025}.
In order to formulate our result we 
recall that the set $\sing(f^{-1})$ of singularities of the inverse 
to a transcendental entire function $f$ consists of the critical
and (finite) asymptotic values of~$f$.
An important class of functions in transcendental dynamics is
the \emph{Eremenko--Lyubich class} $\B$ consisting of all entire functions $f$
for which $\sing(f^{-1})$ is bounded;
see~\cite{Sixsmith2018} for a survey of the dynamics of functions in this class.

An important tool to study this class is the \emph{logarithmic change of variable}.
It was introduced by Eremenko and Lyubich \cite[Section~2]{Eremenko1992} to transcendental dynamics.
We describe it briefly, using 
here and in the following, for $a\in\C$ and $r>0$, the notation
$D(a,r)\defeq \{z\colon |z-a|<r\}$,
$\Delta_r\defeq \{z\colon |z|>r\}$ and
$\H_{>r}\defeq\{z\colon \re z>r\}$.
The half-planes $\H_{<r}$, $\H_{\geq r}$ and $\H_{\leq r}$ are defined analogously.

Let $f\in\B$. Choose $R>|f(0)|$ with $\overline{\sing(f^{-1})}\subset D(0,R)$.
Put $V_R\defeq f^{-1}(\Delta_R)$ and $W_R\defeq \exp^{-1}(V_R)$.
Eremenko and Lyubich~\cite{Eremenko1992}
showed that then there exists a $2\pi i$-periodic function 
\begin{equation} \label{a0}
F\colon W_R\to \H_{>\log R}
\end{equation}
which maps every connected component of $W_R$ biholomorphically to $\H_{>\log R}$
and satisfies $\exp F(z)=f(\exp z)$ for $z\in W_R$.
This function $F$ is called the \emph{logarithmic transform} of~$f$.

An important tool to study the class $\B$ is the estimate~\cite[Lemma~1]{Eremenko1992}
\begin{equation} \label{a1}
|F'(w)|\geq \frac{1}{4\pi} (\re F(w) - \log R)
\quad \text{for}\  w\in W_R.
\end{equation}
This was used by Eremenko and Lyubich~\cite[Theorem~1]{Eremenko1992}
to prove that $I(f)\subset J(f)$ for $f\in\B$.

A class of entire functions satisfying a stronger condition
than~\eqref{a1} was studied in~\cite{Bergweiler2025}.
It was assumed that there exist $\beta,\rho>0$ such that
\begin{equation} \label{a2}
|F'(w)|\geq \beta |F(w)|
\quad \text{for}\ w\in\H_{>\rho}.
\end{equation}
Passing to a smaller value of $\beta$ if necessary, we may assume that 
this inequality holds for all $w\in W_R$.
We will see in Proposition~\ref{prop1} that it actually suffices to assume that 
it holds for all $w\in \partial W_R$.

If $f\in\B$ satisfies this condition,
then we say that $f$ has \emph{regular logarithmic transform}.
We will explain in Section~\ref{sec:regular} why this condition
can be considered as a regularity condition.

It was shown in~\cite[Proposition~9.1]{Bergweiler2025} that functions $f$ of the form
\begin{equation} \label{fpq}
f(z)=\int_0^z p(t)e^{q(t)} dt +c,
\end{equation}
with polynomials $p$ and $q$, and $c\in\C$, have regular logarithmic transform.
It is easy to see that this is also the case for functions $f$ of the 
form $f(z)=a e^z+be^{-z}$ with $a,b\in\C\setminus\{0\}$.

Recall that $V_R\defeq f^{-1}(\Delta_R)$. The connected components of $V_R$ are 
called \emph{tracts}.
\begin{theorem} \label{thm6}
Let $f\in\B$. Suppose that $f$ has finitely many tracts 
and regular logarithmic transform.
For $\lambda\in\C\setminus\{0\}$, put $f_\lambda(z)\defeq f(\lambda z)$.
Then~\eqref{a3} holds.
Moreover, if $|\lambda|$ is sufficiently small, then
\begin{equation} \label{y5}
\dim\!\left( J(f_\lambda) \setminus I(f_\lambda) \right) 
\leq 1+\frac{\log\log\log \frac{1}{|\lambda|}}{\log\log \frac{1}{|\lambda|}}.
\end{equation}
\end{theorem}
We remark that defining $f_\lambda(z)\defeq \lambda f(z)$
instead of $f_\lambda(z)\defeq f(\lambda z)$ leads to the same conclusion,
since the functions $z\mapsto \lambda f(z)$ and
$z\mapsto f(\lambda z)$ are conjugate to each other.
The examples $f(z)=e^z$ or, more generally, $f(z)=ae^z+be^{-z}$ for which we have~\eqref{y6}
show that the bound given by~\eqref{y5} is essentially best possible.

An entire function $f$ is said to be of \emph{finite order} 
if there exists $\mu>0$ such that $|f(z)|\leq \exp(|z|^\mu)$ if $|z|$ is 
sufficiently large. 
The Denjoy--Carleman--Ahlfors theorem (see \cite[Chapter 5, Theorem 1.2]{Goldberg2008} or
\cite[\S XI.4]{Nevanlinna1953}) says that an entire function of finite order has 
only finitely many tracts.
Thus the hypothesis in Theorem~\ref{thm6}
that $f$ has finitely many tracts is always satisfied if $f$ has finite order.

For $M>0$ we put
\begin{equation} \label{eq:Jgeqrho}
J_{\geq M}(f) \defeq \{z\in J(f) \colon \lvert f^n(z)\rvert\geq M \text{ for all $n\geq 0$}\}.
\end{equation}

We will deduce Theorem~\ref{thm6} from the following result.
\begin{theorem} \label{thm5}
Let $f\in\B$. Suppose that $f$ has finitely many tracts and regular logarithmic transform.
Let $R_0\geq 1$ with $\overline{\sing(f^{-1})}\cup\{f(0)\}\subset D(0,R_0)$
and let $\beta$ be such that~\eqref{a2} holds with $\rho=\log R_0$.
Let $N$ be the number of tracts of~$f$, let $R\defeq \max\{R_0,\exp(4/\beta)\}$ and let $1<t\leq 2$. If
\begin{equation} \label{y3}
M\geq \max\left\{e^{5\pi} R, \exp\!\left(\!\left(\frac{10\cdot 4^t\cdot N}{(t-1)\cdot\beta^t}\right)^{1/(t-1)}\right)\right\},
\end{equation}
then
\begin{equation} \label{y4}
\dim\!\left( J_{\geq M}(f) \setminus I(f) \right) \leq t.
\end{equation}
\end{theorem}
Note that the bound for  $M$ depends only on~$t$, $R_0$, $N$ and $\beta$.

As already mentioned, Urba\'nski and Zdunik
showed that~\eqref{a3} implies that $\dim (J(E_\lambda)\setminus I(E_\lambda))<2$
for $|\lambda|<1/e$. In fact, their proof shows that this
holds if $E_\lambda$ has an attracting fixed point.

We will give an analogous corollary of Theorem~\ref{thm6}.
In order to state this corollary we recall that an entire function $f$ is said to be 
of \emph{disjoint type} if it belongs to $\B$ and if the closure of $\sing(f^{-1})$ is 
contained in the immediate attracting basin of an attracting fixed point.
Since the immediate basin of an attracting fixed point always contains at least one
singularity of the inverse function 
(see \cite[Theorem~7]{Bergweiler1993} or \cite[Theorem~2.3]{Schleicher2010})
and since $\sing(E_\lambda^{-1})=\{0\}$, it follows 
that $E_\lambda$ is of disjoint type if and only if $E_\lambda$
has an attracting fixed point.

\begin{corollary} \label{thm4}
Let $f\in\B$ be of disjoint type.
Suppose that $f$ has finitely many tracts and regular logarithmic transform.
Then
$\dim\!\left( J(f) \setminus I(f) \right) <2$.
\end{corollary}

The restriction to functions with regular logarithmic transform is necessary
in Theorems~\ref{thm6} and~\ref{thm5} and Corollary~\ref{thm4}.
Indeed, Rempe~\cite{Rempe2014} has shown that there exists an entire
function $f$ of finite order and disjoint type for which
$\dim (J(f)\setminus I(f))=2$.


\section{Preliminaries}\label{sec:prelim}
As already noted in the introduction,
Eremenko and Lyubich~\cite[Section~2]{Eremenko1992} used the logarithmic change
of variable to prove that $I(f)\subset J(f)$ for $f\in\B$.
The following result is proved by essentially the same method.
\begin{lemma} \label{la9}
Let $f\in\B$ and $R>0$ with $\overline{\sing(f^{-1})}\cup\{f(0)\}\subset D(0,R)$.
If $z\in\C$ satisfies $|f^n(z)|\geq e^{5\pi} R$  for all $n\geq 0$,
then $z\in J(f)$.
\end{lemma}
We will sketch the proof below.
It will use the Koebe one quarter theorem.
Since later we will also need the Koebe distortion theorem, we state both results here.
\begin{lemma} \label{la-koebe}
Let $g\colon D(a,r)\to\C$ be univalent, $\rho \in (0,1)$ and
$z\in D(a,\rho r)$. Then
\begin{equation}\label{la2a}
\frac{\rho}{(1+\rho)^2}
|g'(a)|r
\leq |g(z)-g(a)|
\leq
\frac{\rho}{(1-\rho)^2}
|g'(a)|r 
\end{equation}
and
\begin{equation}\label{la2b}
\frac{1-\rho}{(1+\rho)^3}
|g'(a)|
\leq |g'(z)|
\leq
\frac{1+\rho}{(1-\rho)^3}
|g'(a)| .
\end{equation}
Moreover,
\begin{equation}\label{la2c}
g(D(a,r))\supset
D\!\left(g(a),\frac14 |g'(a)|r\right).
\end{equation}
\end{lemma} 
In most textbooks (e.g., \cite[Theorems 2.3, 2.5 and~2.6]{Duren1983})
the results are stated only for the case that $a=0$ and $r=1$, but the above 
version easily follows from this.

\begin{proof}[Sketch of proof of Lemma~\ref{la9}] 
Let $z\in \Delta_{e^{5\pi}R}$ such that
$f^n(z)\in \Delta_{e^{5\pi}R}$ for all $n\geq 1$.
With the logarithmic transform $F\colon W_R\to \H_{>\log R}$ as in~\eqref{a0}
and with $z=e^w$ this condition 
takes the form $\re F^n(w)\geq 5\pi+\log R$ for all~$n\geq 0$.
It then follows from~\eqref{a1} that 
$|(F^n)'(w)|\geq (5/4)^n$ for all $n\geq 0$.

If $z\in F(f)$, then there exists $\varepsilon>0$ such that
$f^n(D(z,\varepsilon))\subset V_R=f^{-1}(\Delta_R)$
for all $n\geq 0$. Hence there exists $\delta>0$ 
such that $F^n(D(w,\delta))\subset W_R$ for all $n\geq 0$.
By Koebe's one quarter theorem, $F^n(D(w,\delta))$ contains a disk around $F^n(w)$
of radius $|(F^n)'(w)|\delta/4$, which is at least $5^n\delta/4^{n+1}$.
On the other hand, $W_R$ does not contain a disk of radius greater than~$\pi$.
This is a contradiction if $n$ is large enough.
\end{proof}
\begin{remark}
Rempe~\cite{Rempe2023} proved that the constant $4\pi$ in~\eqref{a1} may
be replaced by~$2$. This yields that the constant $5\pi$ in Lemma~\ref{la9} 
may be replaced by any constant greater than~$2$.
\end{remark}

\section{Proof of Theorems~\ref{thm6} and~\ref{thm5}}\label{sec:proof-thm2}
\begin{lemma} \label{la10}
Let $f$, $R$, $\beta$ and $N$ be as in Theorem~\ref{thm5} and
let $F\colon W_R\to \H_{>\log R}$ be the logarithmic transform as in~\eqref{a0}.

Let $U$ be a connected component of $W_R$, let $L\geq \log R+3$ and let $w\in W_R\cap \H_{\geq L-1}$.
Let $0<r\leq 1$ and $1<t\leq 2$. Then 
\begin{equation} \label{v1}
\bigcup_{k\in\Z} F^{-1}(D(w+2\pi i k,r)) \cap U\cap \H_{\geq L}
\end{equation}
can be covered by disks $D(a_k,r_k)$, $k\in\Z$, such that 
$a_k\in\H_{\geq L-1}$ for all $k$ and
\begin{equation} \label{v2}
\sum_{k\in\Z} r_k^t < \frac{C}{L^{t-1}} r^t ,
\end{equation}
where 
\begin{equation} \label{v3}
C \defeq\frac{5 \cdot 4^t }{(t-1)\beta^t} .
\end{equation}
\end{lemma}
\begin{proof} 
Without loss of generality we may assume that $|\im w|\leq \pi$.
Let $\phi$ be the branch of the inverse of $F$ that maps $\H_{>\log R}$ to~$U$.
Put $w_k\defeq w+2\pi i k$ and $a_k\defeq \phi(w_k)$. 
Since $\re w>L-1\geq \log R+2$,
the function $\phi$ is univalent in the disk $D(w_k,2)$ and hence in $D(w_k,2r)$.
Taking $\rho=1/2$ in Koebe's distortion theorem~\eqref{la2a}
thus yields together with condition~\eqref{a2} that if $a_k\in\H_{>\log R}$, then
\begin{equation}\label{v4} 
\begin{aligned} 
\phi(D(w_k,r)) 
&\subset D\!\left(\phi(w_k),2|\phi'(w_k)|2r\right)
= D\!\left(a_k,\frac{4r}{|F'(a_k)|}\right)
\\ &
\subset D\!\left(a_k,\frac{4r}{\beta|F(a_k)|}\right)
= D\!\left(a_k,\frac{4r}{\beta|w_k|}\right).
\end{aligned}
\end{equation}
Put $r_k\defeq 4r/(\beta|w_k|)$.
Since $|w_k|\geq \re w_k\geq \log R+2\geq 4/\beta$
we have $r_k\leq 1$.
We can thus cover the set~\eqref{v1} with the disks $D(a_k,r_k)$,
taking only those indices $k$ for which $\re a_k>L-1$.
Let $I_L$ be the set of these indices $k$ and let $w=u+iv$. 
Since we assumed that $u>L-1\geq  \log R+2\geq 2$ and $|v|\leq\pi$ we find that 
\begin{equation}\label{v5} 
\begin{aligned} 
\sum_{k\in I_L}
 \frac{1}{|w_k|^t}
&=
\sum_{k\in I_L}
\frac{1}{(u^2+(v+2\pi k)^2)^{t/2}}
\\ &
\leq 
\frac{1}{u^t}
+2 \sum_{k=1}^\infty \frac{1}{(u^2+k^2)^{t/2}}
\leq \frac{1}{u^t} +2\int_0^\infty \frac{ds}{(u^2+s^2)^{t/2}}
\\ &
= \frac{1}{u^t} +\frac{2}{u^{t-1}}\int_0^\infty \frac{dx}{(1+x^2)^{t/2}}
\leq
\frac{1}{u^t} +\frac{4}{u^{t-1}}\int_0^\infty \frac{dx}{(1+x)^{t}}
\\ &
=
\frac{1}{u^t} + \frac{4}{(t-1) u^{t-1}} 
<
\frac{5}{(t-1) L^{t-1}} .
\end{aligned}
\end{equation}
Together with the definition of $r_k$ this yields that 
\begin{equation} \label{v2a}
\sum_{k\in I_L} r_k^t < \frac{C}{L^{t-1}} r^t ,
\end{equation}
with $C$ as in~\eqref{v3}.
Renaming the $a_k$ and $r_k$ if necessary we may actually assume that $I_L=\Z$.
\end{proof}
\begin{proof}[Proof of Theorem~\ref{thm5}] 
Let $F\colon W_R\to \H_{>\log R}$ be the logarithmic transform as described in the introduction,
with $W_R= \exp^{-1}(V_R)$ and $V_R= f^{-1}(\Delta_R)$.
We consider an injective curve $\sigma$ in
$\C\setminus \overline{V_R}$ which connects $0$ to $\infty$.
Let $\Sigma$ be a connected component
of the preimage of $\sigma$ under the exponential function.
For $k\in\Z$ we put $\Sigma_k\defeq \Sigma+2\pi i k$.
Then $\Sigma_k$ and $\Sigma_{k+1}$ bound  a ``strip-like'' domain which we denote
by $S(k)$. To achieve that the $S(k)$ give a partition of the plane we add
$\Sigma_k$ to $S(k)$. However, this will be irrelevant for our later considerations,
since we are interested in points which stay in $W_R$ under iteration of~$F$,
and $W_R\cap \Sigma_k=\emptyset$ for all $k\in\Z$.

An important ingredient in the paper by Urba\'nski and Zdunik~\cite{Urbanski2003}
is to pass from the function $E_\lambda$ to the function $\pi\circ E_\lambda$,
where $\pi$ is the projection onto the strip $\{z\colon -\pi<\im z\leq\pi\}$.
Analogously we consider the projection
$\pi_0\colon \C\to S(0)$ defined by $\pi_0(z)=z-2\pi i k$ if $z\in S(k)$
and study the function
\begin{equation} \label{p10}
F_0\colon S(0)\cap W_R\to S(0), \quad
F_0(z)=\pi_0(F(z)).
\end{equation}

We already noted in the proof of Lemma~\ref{la9} that for $w\in W_R$ and
$z=e^w$ we have $|f^n(z)|\geq M$ if and only if $\re F^n(w)\geq L\defeq \log M$.
And this lemma yields that if this is satisfied, then $z\in J(f)$. 
With
\begin{equation} \label{v7}
I^*(F)\defeq 
\left\{w\in W_R\colon  \lim_{n\to\infty}\re F^n(w)=\infty\right\} 
\end{equation}
and
\begin{equation} \label{v8}
J^*_{\geq L}(F)\defeq 
\left\{w\in W_R\colon \re F^n(w)\geq L\ \text{for all}\ n\geq 0 \right\} 
\end{equation}
it thus suffices to prove that 
$\dim\!\left(J^*_{\geq L}(F) \setminus I^*(F)\right)\leq t$.

Let $V_{R,1},\dots,V_{R,N}$ be the connected components of~$V_R$ and,
for $1\leq j\leq N$, let $W_{R,j}$ be the connected component of
$\exp^{-1}(V_{R,j})$ contained in $S(0)$. Then
\begin{equation} \label{v9}
W_R=\bigcup_{j=1}^N\bigcup_{k\in\Z}
(W_{R,j}+2\pi i k) .
\end{equation}
Let $W^0\defeq\bigcup_{j=1}^N W_{R,j}=W_R\cap S(0)$.

By the periodicity of $F$ we have $F_0^n=\pi_0\circ F^n|_{W^0}$ for all $n\geq 0$. 
With $J^*_{\geq L}(F_0)$ and $I^*(F_0)$ defined analogously 
to~\eqref{v7} and~\eqref{v8} we thus find that
$\dim(J^*_{\geq L}(F) \setminus I^*(F))=\dim(J^*_{\geq L}(F_0) \setminus I^*(F_0))$.

The condition~\eqref{y3} for $M$ implies that $L^{t-1}=(\log M)^{t-1}\geq 2NC$, where $C$ is the constant defined in~\eqref{v3}.
Let now $w\in W^0$ with $\re w\geq L-1$ and let $0<r\leq 1$ with $D(w,r)\subset W^0$.
Then
\begin{equation} \label{v10}
\begin{aligned}
F_0^{-1}(D(w,r))\cap \H_{\geq L} 
&=\bigcup_{k\in\Z} F^{-1}(D(w+2\pi i k,r)) \cap W^0\cap \H_{\geq L}
\\ &
=\bigcup_{j=1}^N \bigcup_{k\in\Z} F^{-1}(D(w+2\pi i k,r)) \cap W_{R,j}\cap \H_{\geq L} .
\end{aligned}
\end{equation}
Lemma~\ref{la10} says that for each $j$ the set
\begin{equation} \label{v10a}
\bigcup_{k\in\Z} F^{-1}(D(w+2\pi i k,r)) \cap W_{R,j}\cap\H_{\geq L}
\end{equation}
can be covered by disks $D(a_{j,k},r_{j,k})$ with $a_{j,k}\in\H_{\geq L-1}$
for all $j$ and $k$ such that
\begin{equation} \label{v10b}
\sum_{k\in\Z} r_{j,k}^t\leq \frac{C}{L^{t-1}}r^t\leq \frac{1}{2N} r^t.
\end{equation}
Thus 
\begin{equation} \label{v11}
F_0^{-1}(D(w,r))\cap \H_{\geq L}
\subset 
\bigcup_{j=1}^N \bigcup_{k\in\Z} D(a_{j,k},r_{j,k})
\end{equation}
and 
\begin{equation} \label{v12}
\sum_{j=1}^N \sum_{k\in\Z} r_{j,k}^t \leq \frac12 r^t.
\end{equation}
In particular, this implies that $r_{j,k}\leq r\leq 1$ for all $j$ and~$k$.
We may thus apply Lemma~\ref{la10} with $D(w,r)$ replaced by $D(a_{j,k},r_{j,k})$.
We obtain a covering of 
$F_0^{-2}(D(w,r))\cap J^*_{\geq L}(F_0)$ with disks such that the sum
of the radii to the power $t$ is at most $r^t/4$.
Inductively we obtain, for each $n\geq 1$, a covering of
$F_0^{-n}(D(w,r))\cap J^*_{\geq L}(F_0)$ with disks $D(c_k,\rho_k)$ 
such that $\sum_{k\in\Z} \rho_k^t\leq r^t/2^n$.

We conclude that the set of all $w\in J^*_{\geq L}(F_0)$ whose orbit with respect
to $F_0$ intersects $D(w,r)$ infinitely often has Hausdorff dimension at most~$t$.
Since every compact subset $K$ of $W^0\cap\H_{\geq L}$ can be 
covered by finitely many such disks, we find that 
the set of all $w\in J^*_{\geq L}(F_0)$ whose orbit intersects $K$ 
infinitely often has Hausdorff dimension at most~$t$. We conclude that
$\dim(J^*_{\geq L}(F_0) \setminus I^*(F_0))\leq t$.
\end{proof}
\begin{proof}[Proof of Theorem~\ref{thm6}] 
We show first that if $M\geq e^{5\pi}R$ is large enough, then
the condition~\eqref{y3} is satisfied for
\begin{equation} \label{y7}
t= 1+ \frac{\log\log\log M }{\log\log M}.
\end{equation}
In order to do so we may assume that the constant $\beta$ in~\eqref{a2} satisfies $\beta\leq 1$.
With $t$ given by~\eqref{y7} and $K\defeq 160 N/\beta^2$ we find that 
if $\log\log\log M\geq K$, then
\begin{equation} \label{y8}
\begin{aligned} 
\log\!\left(\!\left(\frac{10\cdot 4^t\cdot N}{(t-1)\cdot\beta^t}\right)^{1/(t-1)}\right)
&\leq \log\!\left(\!\left(\frac{K}{t-1}\right)^{1/(t-1)}\right)
= \frac{1}{t-1} \log \frac{K}{t-1}
\\ &
= \frac{\log\log M}{\log\log\log M} \log \frac{K\log\log M}{\log\log\log M}
\leq \log\log M
\end{aligned} 
\end{equation}
and thus 
\begin{equation} \label{y8a}
\exp\!\left(\!\left(\frac{10\cdot 4^t\cdot N}{(t-1)\cdot\beta^t}\right)^{1/(t-1)}\right)
\leq  M .
\end{equation}
Thus~\eqref{y3} holds for $t$ given by~\eqref{y7} if $M$ is sufficiently large.

We will apply this with $M\defeq 1/|\lambda|$.
If $|\lambda|$ is sufficiently small and thus $M$ is sufficiently large,
we have
$f_\lambda(D(0,M))=f(D(0,|\lambda| M))=f(D(0,1))\subset D(0,M)$.
For such $\lambda$ we find that $D(0,M)\subset F(f_\lambda)$. 
This implies that $J(f_\lambda)=J_{\geq M}(f_\lambda)$.
Theorem~\ref{thm5} thus yields that 
\begin{equation} \label{y9}
\dim\!\left( J(f_\lambda) \setminus I(f_\lambda) \right)
=\dim\!\left( J_{\geq M}(f_\lambda) \setminus I(f_\lambda) \right)
\leq 1+t
\end{equation}
for $t$ given by~\eqref{y7}. Hence~\eqref{y5} follows.

To prove~\eqref{a3} we use a result of Bara\'nski, Karpi\'nska and Zdunik~\cite{Baranski2009}.
They have shown that the hyperbolic dimension of a function in $\B$ is strictly greater than~$1$.
In particular, this implies that
$\dim\!\left( J(f_\lambda) \setminus I(f_\lambda) \right)>1$ for all~$\lambda\neq 0$.
Together with~\eqref{y5} this yields~\eqref{a3}.
\end{proof}
\begin{proof}[Proof of Corollary~\ref{thm4}] 
Let $f_\lambda(z)=f(\lambda z)$ be as in Theorem~\ref{thm6}.
If $|\lambda|$ is sufficiently small, then $f_\lambda$ is of disjoint type.
By hypothesis, $f$ is also of disjoint type.
It now follows from a result of Rempe (\cite{Rempe2009}, 
see also \cite[Proposition~8.16]{Bergweiler2025a})
that there exists a quasiconformal homeomorphism $\phi\colon\C\to\C$ which is
a conjugacy between $f_\lambda$ and $f$ on their Julia sets.
Thus $J(f) \setminus I(f)=\phi( J(f_\lambda) \setminus I(f_\lambda))$.

By Theorem~\ref{thm6}
 we have $\dim\!\left( J(f_\lambda) \setminus I(f_\lambda) \right) <2$ if $|\lambda|$
is sufficiently small.
A result of Gehring and V\"ais\"al\"a~\cite[Theorem~12]{Gehring1973} now yields that
$\dim\!\left( J(f) \setminus I(f) \right) <2$.
\end{proof}
\begin{remark}
A famous result of Astala~\cite{Astala1994} 
gives the sharp bound for the distortion of Hausdorﬀ dimension under quasiconformal mappings.
For our purposes the weaker result in~\cite[Theorem~12]{Gehring1973} suffices.
\end{remark}

\section{Some remarks on functions with regular logarithmic transform}\label{sec:regular}
The following proposition was already mentioned in the introduction.
\begin{proposition}\label{prop1}
Let $f\in\B$ and $R>|f(0)|$ with $\overline{\sing(f^{-1})}\subset D(0,R)$.
Let $F\colon W_R\to \H_{>\log R}$ be the logarithmic transform as in~\eqref{a0}.
Let $\beta>0$ and suppose that $|F'(z)|\geq \beta |F(z)|$ for all $z\in \partial W_R$.
Then $|F'(z)|\geq \beta |F(z)|$ for all $z\in W_R$.
\end{proposition}
The proof uses the following result of Lewis, Rossi and Weitsman~\cite{Lewis1984},
which generalized a result of Huber~\cite{Huber1957}.
\begin{lemma}\label{lemma-huber}
Let $u\colon\C\to[-\infty,\infty)$ be subharmonic and suppose that
\begin{equation}\label{hu1}
\lim_{r\to\infty}\frac{\max_{|z|=r} u(z)}{\log r}=\infty.
\end{equation}
Then there exists a path $\gamma$ tending to $\infty$ such that
\begin{equation}\label{hu4}
\int_{\gamma} e^{-\lambda u(z)}|d z|<\infty
\quad\text{for each } \lambda>0
\end{equation}
and
\begin{equation}\label{hu5}
\frac{u(z)}{\log |z|}\to \infty
\quad\text{as } z\to\infty \text{ along } \gamma.
\end{equation}
\end{lemma}
\begin{proof}[Proof of Proposition~\ref{prop1}]
Let $T$ be a connected component of $W_R$.
Suppose that $|F'(z)|< \beta |F(z)|$ for some $z\in T$.
Define $u\colon\C\to\R$ by $u(z)=\log|F(z)/F'(z)|$ if $z\in T$ and $|F'(z)|< \beta |F(z)|$,
and $u(z)=\log(1/\beta)$ otherwise. Then $u$ is a non-constant subharmonic function.

Since $u$ is bounded on $\partial T$, it follows from Wiman's theorem in its version
for subharmonic functions \cite[Theorem 6.4]{Hayman1989} that
\begin{equation}\label{wi1}
\liminf_{r\to\infty}\frac{\max_{|z|=r} u(z)}{\sqrt{r}}>0.
\end{equation}
Thus~\eqref{hu1} is satisfied.

Let $\gamma$ be the path provided by Lemma~\ref{lemma-huber}. Taking $\lambda=1$ we deduce that
\[
\int_{\gamma} \left| \frac{F'(z)}{F(z)}\right| |d z|<\infty.
\]
Hence the image of $\gamma$ under a branch of $\log F$ has finite Euclidean length.
It follows that this branch tends to some value $a\in\C$ as $z$ tends to $\infty$ along~$\gamma$.
We conclude that $F(z)$ tends to $b\defeq e^a$ as $z$ tends to $\infty$ along $\gamma$.
Thus $f$ tends to $c\defeq e^b$ along $\exp\circ\gamma$.
Note here that $\exp\circ\gamma$ is also a curve tending to~$\infty$,
since if $|z|\to\infty$ in $T$, then also $\re z\to\infty$.
Thus $c$ is an asymptotic value of $f$.

Since $\re F(w)>\log R$ for $w\in W_R$ it follows that $\re b\geq\log R$. Thus $|c|\geq R$.
This is a contradiction, since $c\in\sing(f^{-1})\subset D(0,R)$ by the choice of~$R$.
\end{proof}

Proposition~\ref{prop1} says that~\eqref{a2} is equivalent to the condition
\begin{equation} \label{a2x}
|F'(w)|\geq \beta |F(w)|
\quad \text{for}\ w\in \partial W_R,
\end{equation}
possibly with a different value of $\beta$.
We will give some explanation why~\eqref{a2x} can be considered as a ``regularity condition''.

Let $T$ be a connected component of $W_R$.
Then $z\mapsto\log (F(z)-\log R)$ maps $T$ onto a parallel strip of width $\pi$.
The Ahlfors distortion theorem
(see, e.g., \cite[p.\ 98]{Nevanlinna1953})
 yields that
\begin{equation}\label{hu2}
\log |F(z)|\geq \frac12 \re z-O(1)
\end{equation}
 as $\re z\to\infty$.

Fix a point $z_0\in \partial T$. For $z\in \partial T$ let
$\Gamma_z$ be the part of $\partial T$ that connects $z_0$ with $z$.  Then
\begin{equation}\label{hu3}
\log |F(z)|=\int_{\Gamma_z} \left| \frac{F'(\zeta)}{F(\zeta)}\right| |d \zeta|+O(1)
\end{equation}
 as $\re z\to\infty$.
Combining \eqref{hu2} and \eqref{hu3}
shows that the ``average value'' of $|F'(z)/F(z)|$ over the curve $\Gamma_z$
is at least $\re z/(2\length(\Gamma_z))$.

We conclude that $|F'(z)|\geq \beta |F(z)|$ holds under the following conditions:
\begin{itemize}
\item[(a)]
$\partial T$ is not too wiggly in the sense that $\length(\Gamma_z)=O(\re z)$;
\item[(b)]
$|F'(z)/F(z)|$ does not oscillate too much on $\partial T$.
\end{itemize}
We note that Rottenfu{\ss}er, R\"uckert, Rempe and Schleicher \cite[Definition~5.3]{Rottenfusser2011}
introduced the concept of ``bounded wiggling'' of a tract.
Condition (a) is different from their definition, even though the underlying idea is similar.

\bibliography{literature}{}
\bibliographystyle{plain} 

\noindent
Walter Bergweiler\\
Mathematisches Seminar\\
Christian-Albrechts-Universit\"at zu Kiel\\
Heinrich-Hecht-Platz 6\\
24098 Kiel\\
Germany\\
{\tt Email: bergweiler@math.uni-kiel.de}

\medskip

\noindent
Weiwei Cui and Lingrui Wang\\
Research Center for Mathematics and Interdisciplinary Sciences\\
Shandong University\\
Qingdao, 266237\\
China\\
{\tt Email: weiwei.cui@sdu.edu.cn, lrwang@sdu.edu.cn}
\end{document}